\documentclass[11pt]{amsart}

\usepackage{enumerate,graphics}
\usepackage[latin1]{inputenc}

\DeclareMathOperator{\End}{End}
\DeclareMathOperator{\Hom}{Hom}

\DeclareMathOperator{\SU}{SU}

\DeclareMathOperator{\Span}{Span}
\DeclareMathOperator{\Id}{Id}

\DeclareMathOperator{\dbar}{\bar \partial}

\renewcommand{\Im}{\operatorname{Im}}
\renewcommand{\Re}{\operatorname{Re}}

\newcommand{\R}{\mathbb{R}}
\newcommand{\C}{\mathbb{C}}

\newcommand{\Z}{\mathbb{Z}}
\renewcommand{\H}{\mathbb{H}}  

\newcommand{\HP}{\mathbb{HP}}

\newcommand{\ii}{i}        
\newcommand{\jj}{j}
\newcommand{\kk}{k}

\numberwithin{equation}{section}  

\theoremstyle{plain}
\newtheorem{The}[subsection]{Theorem}
\newtheorem{Pro}[subsection]{Proposition}
\newtheorem{Lem}[subsection]{Lemma}
\newtheorem{Cor}[subsection]{Corollary}
\newtheorem*{The*}{Theorem}

\theoremstyle{definition}

\theoremstyle{remark}

\usepackage[
  breaklinks, 
  colorlinks=true, 
  linkcolor=black, 
  anchorcolor=black,
  citecolor=black,
  filecolor=black,
  menucolor=black,
  urlcolor=blue,
  bookmarks,
  bookmarksnumbered,
  hyperfootnotes=false,
  pdfpagelabels,
  pdfstartview=FitH, 
  pdfpagemode=UseOutlines, 
  pdfnewwindow=true,
  pagebackref=false, 
  bookmarksopen=false,
  bookmarks=true
]{hyperref}
\usepackage[figure]{hypcap}

\begin{document}
\title{Euclidean minimal tori with planar ends and elliptic solitons}

\author{Christoph Bohle}
\author{Iskander A.\ Taimanov}

\address{Christoph Bohle\\
  Universit\"at T\"ubingen\\
  Mathematisches Institut \\
  Auf der Morgenstelle 10\\
  72076 T\"ubingen \\
  Germany}

\address{Iskander A.\ Taimanov\\
Sobolev Institute of Mathematics\\ 630090
    Novosibirsk\\ Russia}

\email{bohle@mathematik.uni-tuebingen.de \\ taimanov@math.nsc.ru}


\date{\today}

\begin{abstract}
  A Euclidean minimal torus with planar ends gives rise to an immersed
  Willmore torus in the conformal 3--sphere $S^3=\R^3\cup \{\infty\}$.
  The class of Willmore tori obtained this way is given a spectral
  theoretic characterization as the class of Willmore tori with
  reducible spectral curve.  A spectral curve of this type is
  necessarily the double of the spectral curve of an elliptic KP
  soliton. The simplest possible examples of minimal tori with planar
  ends are related to 1--gap Lam\'e potentials, the simplest
  non--trivial algebro geometric KdV potentials.  If one allows for
  translational periods, Riemann's ``staircase'' minimal surfaces
  appear as other examples related to 1--gap Lam\'e potentials.
\end{abstract}

\thanks{The first author (C.B.) was supported by DFG Sfb/Tr 71
  ``Geometric Partial Differential Equations'', the second author
  (I.A.T.) was supported by grant RFBR 12-01-00124-a of the program
  ``Fundamental Problems of Nonlinear Dynamics in Mathematical and
  Physical Sciences'' of the Presidium of RAS. In addition both
  authors were supported by the Hausdorff Institute of Mathematics in
  Bonn.  }

\maketitle

\section{Introduction}

Complete minimal surfaces with finite total curvature and planar ends in
Euclidean 3--space can be compactified by filling in points at the ends if one
views them as immersions into the conformal 3--sphere $S^3=\R^3\cup
\{\infty\}$. This is equivalent to compactifying their preimage under
stereographic projection.  The main interest in this M\"obius geometric
compactification stems from the fact that the resulting immersions of compact
surfaces are critical points of the Willmore energy.  In fact, as proven by
Bryant \cite{Br84,Br88}, all Willmore spheres in the conformal 3--sphere can
be obtained from this construction. This is not anymore true for Willmore
immersions of genus $g\geq 1$. For example, the Clifford torus has Willmore
energy $W=2\pi^2$, while compactifications of Euclidean minimal surfaces with
planar ends always have Willmore energy $W=4\pi n$ for $n$ the number of
ends.

In the present paper we characterize Willmore tori in conformal
3--space~$S^3$ that are Euclidean minimal with planar ends for some
point $\infty\in S^3$ at infinity in terms of spectral and integrable
systems theory.  Our integrable systems approach leads to a simple
description of the Euclidean minimal tori previously studied by Costa
\cite{Co} and Kusner, Schmitt \cite{KS} which have $W=16\pi$ and four
ends, the least number of ends possible for Euclidean minimal tori
with planar ends.  It equally applies to minimal tori with planar ends
and translational periods like Riemann's ``staircase'' minimal
surfaces, see e.g.\ \cite{HKR,LRW,MPR,MP}. From the spectral theory
point of view all these examples turn out to be related to the
simplest non--trivial algebro geometric KdV potentials, the 1--gap
Lam\'e potentials.  This observation also sheds new light on the recent
characterization of Riemann's minimal surfaces obtained by Meeks,
Perez and Ros \cite{MPR,MP} which uses algebro geometric solutions to
the KdV equation in an essential way.

The way that spectral and integrable systems theory intervenes in our
setting is through the \emph{spectral curve} of conformally immersed
tori.  The spectral curve is an invariant of tori immersed into
3--dimensional space first considered in \cite{Ta98}. It is defined as
the Riemann surface normalizing the Floquet spectrum of a
2--dimensional Dirac operator
\begin{equation}
  \label{eq:dirac-op}
  D=
\begin{pmatrix}
  0 & \partial \\  -\bar\partial  & 0
\end{pmatrix} + \begin{pmatrix}
  U  & 0 \\ 0 & \bar U
\end{pmatrix}.
\end{equation}
attached to an immersion. The spectral curve is asymptotic to the
``vacuum'' spectrum belonging to the operator with zero potential
$U=0$ which is the disjoint union of two copies of $\mathbb{C}$ with a
$\mathbb{Z}^2$--lattice of double points. For a generic immersion the
spectral curve has infinite genus, because infinitely many of the
vacuum double points turn into handles.  In the special case that only
finitely many handles appear in the asymptotics, the spectral curve
has finite genus and the immersion can be constructed using finite gap
integration.  This happens for example, if the immersion is the
solution to elliptic variational problems related to the area
\cite{PS89,Hi} or Willmore functional~\cite{B}.

The fact that all previously known examples of spectral curves occurring in
surface theory were irreducible made people expect that irreducibility should
hold for general conformally immersed tori (see e.g.\ the ``Pretheorem''
in~\cite{Ta04}).  In the present paper we show that this is not true and
prove (see Theorem~\ref{the:main} below):

\begin{The*}
  Every Euclidean minimal torus $f\colon T^2\rightarrow \R^3\cup \{\infty\}$
  with planar ends has a reducible spectral curve. 
\end{The*}

By general spectral curve theory (cf.\
Appendix~\ref{sec:spectral_curve}), reducibility of the spectral curve
implies in particular that the spectral curve has two irreducible
components of finite genus which are interchanged under an
anti--holomorphic involution. Moreover, each component is a compact
Riemann surface with one puncture. As a consequence, each component of
the spectral curve is the spectral curve of an elliptic KP soliton in
the sense of \cite{Kr} (cf.\ Appendix~\ref{app:elliptic_KP}).  This
makes contact to the theory of elliptic Calogero--Moser systems
\cite{AMM,Kr}.

Combining our result with the fact that a Willmore torus in the
conformal 3--sphere $S^3$ that is not Euclidean minimal with planar
ends has an irreducible spectral curve of finite genus (see
Theorem~5.1 and Corollary~5.3 of \cite{B}), we obtain the following
reformulation of the above theorem:

\begin{The*}
  Every immersed Willmore torus $f\colon T^2\rightarrow S^3$ in the conformal
  3--sphere $S^3$ has finite spectral genus. The spectral curve of $f$ is
  reducible if and only if $f$ is Euclidean minimal with planar ends with
  respect to the Euclidean geometry defined by some point $\infty\in S^3$ at
  infinity.
\end{The*}

To our knowledge, Dirac operators \eqref{eq:dirac-op} corresponding to
Euclidean minimal tori with planar ends are the first known examples
of 2--dimensional periodic Dirac operators with smooth potential and
reducible Floquet spectral curve. (Such Dirac potentials are indeed
globally smooth, because Euclidean minimal tori with planar ends
extend to globally smooth immersions into the conformal 3--sphere
$S^3=\R^3\cup \{\infty\}$.)

In the following we sketch our argument why Euclidean minimal tori
with planar ends have reducible spectral curves.  A conformally
immersed torus $f\colon T^2\rightarrow S^3$ gives rise to Dirac type
operators with smooth potentials in two different ways. One is
M\"obius invariant (Section~\ref{sec:moebius-inv}), the other one
depends on the Euclidean geometry defined by the choice of a point
$\infty\in S^3$ at infinity \emph{not lying on the image of the
  immersion} (Section~\ref{sec:weierstrass}). By
Theorem~\ref{the:moebius_inv_of_spec}, all different Dirac type
operators obtained this way from a conformally immersed torus give
rise to the same spectral curve, the spectral curve of the immersion~$f$.

A general point $\infty\in S^3$ at infinity that \emph{lies on the
  image of the immersion} gives rise (as in
Section~\ref{sec:weierstrass}) to a Dirac type operator with a
non--smooth potential. However, in the special case that the resulting
immersion into Euclidean 3--space is minimal (and then necessarily has
planar ends), the corresponding Dirac operator has a trivial
potential, i.e., is of the form \eqref{eq:dirac-op} with $U=0$.

The main result of the paper (Theorem~\ref{the:main}) is derived from
the fact that the spectral curve of a Euclidean minimal torus with
planar ends can be computed using this Dirac operator
\eqref{eq:dirac-op} with trivial potential $U=0$ if one imposes
suitable boundary conditions at the ends. The reducibility of the
spectral curve then reflects the decomposition of the Dirac operator
into pure $\dbar$-- and $\partial$--operators. In particular, the
components of the spectral curve coincide (Corollary \ref{cor:main})
with spectral curves of $\dbar$--operators on punctured elliptic
curves as discussed in \cite{BT1}

\section{The spectral curve of a Euclidean minimal torus with planar
  ends}\label{sec:spec_minimal}

In the proof of our main result we make use of the quaternionic
approach to conformal surface geometry, see
Appendix~\ref{app:quat_hol}, and treat minimal tori with planar ends
in Euclidean 3--space as conformal immersions $f\colon T^2 \rightarrow
\HP^1$ with values in the conformal 3--sphere $\R^3\cup
\{\infty\}\subset \HP^1$, where $\R^3=\Im\H$ is identified with
$\{[x,1]\mid x\in \Im\H\}\subset \HP^1$ and $\infty = [1,0]$.

We call an immersion $f\colon M\rightarrow \R^3\cup \{\infty\}$ of a
compact surface $M$ \emph{Euclidean minimal with planar ends} if it is
globally smooth and the immersion $f\colon M \backslash
\{p_1,...,p_n\} \rightarrow \R^3$ is Euclidean minimal, where
$\{p_1,...,p_n\}$ denotes the finite set of \emph{ends} at which $f$
goes through $\infty$. The minimal immersion $f\colon M \backslash
\{p_1,...,p_n\} \rightarrow \R^3$ is then complete, has finite total
curvature, and planar ends and, conversely, every such immersion can
be compactified to an immersion into the conformal 3--sphere,
see~\cite[pp.~44--49]{Br84} which moreover proves:

\begin{Lem}\label{lem:minimal_planar}
  A map $f\colon M \rightarrow \R^3\cup \{\infty\}$ defined on a
  compact Riemann surface $M$ is conformal and Euclidean minimal with
  planar ends $\{p_1, ..., p_n\}$ if and only if the $\C^3$--valued
  1--form $\partial f$ is null\footnote{Here $\partial f :=
    \frac12(df-i*df)$, where $*$ denotes the induced complex structure
    on $T^*M$ and $i$ stands for the complex structure of the
    complexification $\C^3$ of $\R^3$; the fact that $\partial f$ is
    null with respect to the complex bilinear extension of the
    Euclidean metric reflects conformality of the immersion.},
  holomorphic, and nowhere vanishing on $M \backslash \{p_1, ..., p_n\}$
  and has second order poles without residues at the
  $\{p_1, ..., p_n\}$.
\end{Lem}

We describe now how the spinorial Weierstrass representation of
minimal surfaces in $\R^3$ appears in the quaternionic framework.
Recall that for an immersion $f\colon M\rightarrow \R^3\cup\{\infty\}$
into the conformal 3--sphere, the corresponding quaternionic line
subbundle $L\subset V$ of the trivial rank 2 bundle $V$ carries a
M\"obius invariant structure of a quaternionic spin bundle with
compatible $\dbar$--operator (Appendix~\ref{sec:weierstrass}). The
fact that the immersion $f$ is Euclidean minimal with planar ends is
equivalent to the fact that the Euclidean quaternionic holomorphic
structure induced by $\infty$ (Appendix~\ref{sec:weierstrass}) has
vanishing Hopf field $Q$ and coincides with the underlying
$\dbar$--operator (the reason being that the Hopf field of a Euclidean
holomorphic line bundle coincides with the mean curvature half
density, see e.g.\ \cite{PP}, and hence vanishes precisely if the
immersion is minimal).  Note that, unlike the M\"obius invariant
underlying $\dbar$--operator, the quaternionic holomorphic section
$\psi$ appearing in the Weierstrass representation is only defined
away from the ends $\{p_1,...,p_n\}$ at which the immersion goes
through $\infty$.

In the case that $M=T^2=\C/\Gamma$ is a torus, the canonical bundle is
holomorphically trivialized $K\cong \underline{\C}$ by the
differential $dz$ of a uniformizing coordinate. The complex spin
bundle $E$ underlying $L$ thus has a nowhere vanishing
$\dbar$--holomorphic section $\varphi$ with $\Z_2$--monodromy $h_0\in
\Hom(\Gamma,\Z_2)$ satisfying $(\varphi,\varphi) = \jj dz$. The
quaternionic holomorphic section $\psi$ appearing in the Weierstrass
representation (Appendix~\ref{sec:weierstrass}) then takes the form
$\psi=\varphi( s_1+ i s_2 \jj)$, where $s_1$, $s_2$ are holomorphic
functions with $\Z_2$--monodromy~$h_0$ defined away from the
ends. Hence
\begin{equation}
  \label{eq:weierstrass}
  df= (\psi,\psi) = (\jj s_1 + \ii s_2) dz (s_1 + \ii s_2 \jj)= \jj ( s_1^2\,
  dz -\bar s_2^2 \, d\bar z) + 2\ii \Re(s_1 s_2 \, dz)
\end{equation}
and, with respect to the basis $\ii$, $\jj$, $\kk$ of $\R^3=\Im\H$, we
obtain the spinorial\footnote{If $M$ is not a torus, but an arbitrary
  Riemann surface, the argument of this paragraph only holds locally;
  if $K$ is not globally trivialized, a global formula can be obtained
  by absorbing $\sqrt{dz}$ and the $\Z_2$--monodromy into $s_1$, $s_2$
  and viewing them as spinor fields instead of functions (then $dz$
  disappears in \eqref{eq:classical_weierstrass} and $\varphi$ further
  above has to be divided by $\sqrt{dz}$), see e.g.\ \cite{Bob,KS} for
  this globalized version.}  Weierstrass representation
\begin{equation}
  \label{eq:classical_weierstrass}
  df =
  \Re \begin{pmatrix}
    2 s_1 s_2\,  dz  \\ (s_1^2 - s_2^2)\,  dz \\ i (s_1^2 + s_2^2)\, dz 
\end{pmatrix}
\end{equation}
for minimal surfaces, see \cite{Wei} for the original and
\cite{Bob,KS} for contemporary, coordinate independent
versions. Because $\partial f$ in Lemma~\ref{lem:minimal_planar}
coincides (up to a factor 1/2) with the $\C^3$--valued 1--form in
\eqref{eq:classical_weierstrass}, we obtain:

\begin{Cor}\label{cor:spinors}
  A map $f\colon T^2\backslash\{p_1,...,p_n\}\rightarrow \R^3$ is
  minimal with planar ends $p_1$,...,$p_n$ if and only if the
  meromorphic functions $s_1$, $s_2$ in
  \eqref{eq:classical_weierstrass} have poles of order at most one at
  the ends $p_1$,...,$p_n$, vanishing order zero terms in their
  Laurent expansions at the ends, and no common zeroes.
\end{Cor}

Conversely, if one starts with meromorphic functions $s_1$, $s_2$ on
$T^2=\C/\Gamma$ with identical $\Z_2$--monodromy and zero and pole
behavior as in the corollary, then \eqref{eq:classical_weierstrass}
defines the differential of a possibly double periodic minimal torus
with planar ends.  The vanishing of the translational periods in
direction of $\gamma\in \Gamma$ is then equivalent to
\begin{equation}
  \label{eq:period_condition}
  \int_\gamma s_1 s_2\, dz \in \ii \R \qquad \text{and}
  \qquad \int_\gamma s_1^2\, dz = \int_\gamma \bar{s}_2^2\, d\bar z.
\end{equation}
In order to obtain a closed torus this has to hold for all $\gamma\in
\Gamma$; for a torus with one translational period it has to hold for
one generator of $\Gamma$.

\begin{The}\label{the:main}
  The spectral curve of a Euclidean minimal torus $f\colon T^2 \rightarrow
  \R^3\cup \{\infty\}$ with planar ends is reducible.
\end{The}

Following from general properties of spectral curves, cf.\
Appendix~\ref{sec:spectral_curve}, the spectral curve of $f$ is thus
of the form $\Sigma=\Sigma'\, \dot\cup\, \bar \Sigma'$ for $\Sigma'$ a
Riemann surface of finite genus with one end.  More precisely
$\Sigma'$ is the spectral curve belonging to an elliptic KP soliton,
see Appendix~\ref{app:elliptic_KP}.  Theorem~\ref{the:main} holds more
generally for Euclidean minimal tori with planar ends and
translational monodromy like Riemann's ``staircase'' minimal surfaces, see
Section~\ref{sec:riemann}.

\begin{proof}
  The basic idea behind the proof is similar to that of
  Theorem~\ref{the:moebius_inv_of_spec}. An essential difference is
  that here we chose a point $\infty$ at infinity for which the
  immersion has ends, i.e., goes through $\infty$. As a consequence,
  the corresponding flat connection $\nabla$ on $V/L$ with
  $\nabla\psi_1 =0$ (cf.\ Appendix~\ref{sec:weierstrass}) is only
  defined away from the ends, because the section $\psi_1$ vanishes at
  the ends. Our strategy in order to determine the spectral curve of
  $V/L$ is to derive a characterization of sections of $KV/L\cong L$
  that are defined away from the ends and are of the form $\nabla
  \psi^h$ for $\psi^h$ a holomorphic section of $V/L$ with
  monodromy~$h$.

  As above denote by $\varphi$ a non--trivial $\dbar$--holomorphic
  section with $\Z_2$--monodromy $h_0$ of the complex line bundle $E$
  underlying $L$.  Let $\psi^h$ be a holomorphic section of $V/L$ with
  monodromy $h$.  It can be written as $\psi^h = \psi_1 \chi$, with
  $\chi$ a quaternion valued function with monodromy $h$ defined away
  from the ends.  Using the identification $KV/L\cong L$ via
  $\delta$, its derivative is then
  \[ \nabla \psi^h = \psi_1 d\chi \cong \varphi (\Phi_1+ \Phi_2
  \jj), \] where $\Phi_1$, $\Phi_2$ are complex holomorphic functions
  with monodromy $h\cdot h_0$ and $\bar h\cdot h_0$ defined away from
  the ends (that $\Phi_1$, $\Phi_2$ are complex holomorphic holds,
  because $\nabla \psi^h$ is a holomorphic section of $KV/L\cong L$,
  cf.\ Appendix~\ref{sec:weierstrass}, whose Hopf field $Q$ vanishes
  identically for $f$ Euclidean minimal). From $\nabla \psi_2 =
  -\psi_1 df \cong -\varphi (s_1+ \ii s_2 \jj)$ we thus obtain
  \begin{equation}\label{eq:d-chi}
    d\chi = df(s_1+ \ii s_2 \jj)^{-1}(\Phi_1+ \Phi_2 \jj) = (\jj s_1 +
  \ii s_2)dz(\Phi_1+ \Phi_2 \jj),
  \end{equation}
  where the last equality holds by
  \eqref{eq:weierstrass}.

  We prove now that $\Phi_1$ and $\Phi_2$ obtained from a holomorphic
  section $\psi^h$ with monodromy of $V/L$ have the same pole behavior
  as $s_1$ and $s_2$, i.e., their only possible poles are first order
  poles at the ends and their Laurent expansions at the ends have
  vanishing order zero terms. To prove the first claim, note that near
  every end the holomorphic section $\psi_2 = -\psi_1 f$
  (Appendix~\ref{sec:moebius-inv}) is non--vanishing so that $\psi^h =
  \psi_2 \chi_2$ for a smooth quaternion valued function
  $\chi_2$. Moreover, because $\psi_1 = -\psi_2 f^{-1}$ is holomorphic
  and $df^{-1}$ is non--zero at a planar end of $f$, there is a smooth
  quaternionic function $g_2$ defined by $d\chi_2 = -df^{-1} g_2$
  (this is in fact the quotient construction of
  Appendix~\ref{sec:moebius-inv}). Taking the derivative of $\chi =
  -f\chi_2$ yields
  \begin{equation}
    \label{eq:d-chi-g1}
    d\chi = -df(\chi_2 + f^{-1} g_2) = -df g_1
  \qquad \textrm{ with } \qquad
    g_1:= \chi_2 + f^{-1} g_2. 
  \end{equation}
  Because $g_1$ is smooth and, by \eqref{eq:d-chi}, away from the ends 
  \begin{equation}
    \label{eq:g1}
    g_1 = -(s_1+ \ii s_2 \jj)^{-1}(\Phi_1+
    \Phi_2 \jj),
  \end{equation} the pole behavior of $s_1$ and $s_2$ implies that 
  $\Phi_1$ and $\Phi_2$ at most have first order poles at the ends.

  The fact that $d\chi$ in \eqref{eq:d-chi} has no periods around a
  given end is equivalent to the fact that the closed forms
  \[ \jj s_1 \Phi_1 \, dz + \ii s_2 \Phi_2 \, dz \jj \qquad \textrm{
    and } \qquad \ii s_2 \Phi_1 \, dz - \bar{s_1} \bar{\Phi_2} \,
  d\bar z\] have vanishing residues at that end (i.e., the integrals
  $\frac1{2\pi \ii} \oint{}$ over small loops around the end are
  zero).  Because the Laurent expansions of $s_1$ and $s_2$ have no
  order zero terms at the end, the order zero terms $\Phi_1(0)$ and
  $\Phi_2(0)$ in the Laurent expansions of $\Phi_1$ and $\Phi_2$ have
  to satisfy
  \[
  \begin{pmatrix}
    s_1(-1)  & \ii \bar{s_2}(-1) \\
    \ii s_2(-1) & \bar{s_1}(-1)
  \end{pmatrix}
  \begin{pmatrix}
    \Phi_1(0) \\ \bar \Phi_2(0)
  \end{pmatrix}
  = 0,
  \]
  where $s_1(-1)$ and $s_2(-1)$ denote the residues of $s_1$ and $s_2$
  at the end.  Because at least one of $s_1$ and $s_2$ has a pole at
  the end, the determinant $|s_1(-1)|^2 + |s_2(-1)|^2$ of this
  $2\times 2$ matrix is non--zero and both $\Phi_1$, $\Phi_2$ have
  vanishing order zero terms at the end.

  Thus, every holomorphic section $\psi^h$ of $V/L$ with monodromy $h$
  can be obtained, by integrating \eqref{eq:d-chi}, from meromorphic
  functions $\Phi_1$ and $\Phi_2$ with monodromies $h\cdot h_0$ and
  $\bar h\cdot h_0$ and first order poles with vanishing order zero
  terms at the ends.

  We prove now that conversely every pair of meromorphic functions
  $\Phi_1$ and $\Phi_2$ with the given properties comes from a
  holomorphic section $\psi^h$ of $V/L$ with monodromy $h$.  By
  Lemma~\ref{lem:integration_preserving_monodromy} below, the form
  $d\chi$ obtained by plugging $\Phi_1$ and $\Phi_2$ into
  \eqref{eq:d-chi} can be integrated to a smooth function $\chi$ with
  monodromy $h$ which is defined away from the ends of the immersion.
  It therefore remains to check that the holomorphic section
  $\psi^h=\psi_1\chi= \psi_2 \chi_2$ thus defined smoothly extends
  through the ends. By \eqref{eq:g1}, the asymptotics of $s_1$, $s_2$
  and $\Phi_1$, $\Phi_2$ at the ends implies that the function $g_1$
  is smooth and $dg_1$ vanishes at the ends.  Using
  \eqref{eq:d-chi-g1} and the defining equation of $g_2$, the
  differential of $g_1$ becomes $dg_1 = f^{-1} dg_2$. Hence $dg_2 = f
  dg_1$ is bounded at the ends and smooth elsewhere so that by
  integration $g_2$ is $C^0$ at the ends. But now $d\chi_2 = -df^{-1}
  g_2$ implies that $\chi_2$ is $C^1$ at the ends. The holomorphic
  section $\psi_1 \chi = \psi_2 \chi_2$ is thus $C^1$ at the ends and
  hence smooth by elliptic regularity.

  So far we have shown that holomorphic sections $\psi^h$ of $V/L$
  with monodromy $h$ correspond to pairs of meromorphic functions
  $\Phi_1$ and $\Phi_2$ with monodromies $h\cdot h_0$ and $\bar h\cdot
  h_0$ and prescribed pole behaviors.  Because for generic points of
  the spectral curve the corresponding space of holomorphic sections
  of $V/L$ with monodromy $h$ is complex 1--dimensional
  (Appendix~\ref{sec:spectral_curve}), generically either $\Phi_1$ or
  $\Phi_2$ has to vanish identically. (Otherwise they would give rise
  to a complex 2--dimensional space of holomorphic sections with
  monodromy $h$, because one could separately plug $\Phi_1$ and $0$ or
  $0$ and $\Phi_2$ into \eqref{eq:d-chi}.)

  This shows that the spectral curve has two connected components,
  namely one on which the corresponding holomorphic sections $\psi^h$
  generically have vanishing $\Phi_2$ and one on which $\Phi_1$
  vanishes generically.  Both parts are interchanged under the
  anti--holomorphic involution which is induced by the symmetry
  $\psi^h\mapsto \psi^h\jj$.
\end{proof}

In the proof of Theorem~\ref{the:main} we have derived the following
characterization of multipliers $h$ admitting non--trivial holomorphic
sections of $V/L$ with monodromy $h$:

\begin{Lem}\label{lem:main}
  Let $f\colon T^2 \rightarrow \R^3\,\cup\, \{\infty\}$ be a Euclidean
  minimal torus with planar ends at $p_1$,...,$p_n\in T^2$ and induced
  spin structure corresponding to a $\Z_2$--multiplier $h_0$.  Then
  $V/L$ admits a holomorphic section $\psi^h$ with monodromy $h$ if
  and only if there is a meromorphic function $\Phi$ with monodromy
  $h\cdot h_0$ or $\bar h \cdot h_0$ that has at most first order
  poles and vanishing order zero terms at the $p_1$,...,$p_n\in T^2$.
\end{Lem}

Because the spectral curve is defined as the Riemann surface
normalizing the set of Floquet multipliers $h$ that belong to
non--trivial holomorphic sections of $V/L$, it can be computed by the
following corollary.

\begin{Cor}\label{cor:main}
  The set of multipliers $h\in \Hom(\Gamma,\C_*)$ admitting a
  non--trivial meromorphic function $\Phi$ with monodromy $h$ on
  $T^2=\C/\Gamma$ such that 
  \begin{itemize}
  \item[a)] all poles are of first order and located at the ends
    $p_1$,...,$p_n\in T^2$ and
  \item[b)] the Laurent series at the ends have vanishing order zero
    terms
  \end{itemize}
  is a 1--dimensional complex analytic set.  Its normalization is one
  connected component of the spectral curve of the minimal torus with
  planar ends and hence a compact Riemann surface with one puncture.
\end{Cor}


In Theorem~\ref{the:main} and the following discussion we derive
properties of spectral curves (like number of ends and components and
finiteness of the genus) of Euclidean minimal tori with planar ends
from the general theory of spectral curves of immersed tori in the
conformal 3--sphere (Appendix~\ref{sec:spec-immersion}), a special
case of the theory of spectral curves for periodic 2--dimensional Dirac
operators (Appendix~\ref{sec:spectral_curve}).

Corollary~\ref{cor:main} indicates an alternative way to study
spectral curves of Euclidean minimal tori with planar ends. This is
further discussed in \cite{BT1}, where we more generally define
spectral curves of $\dbar$--operators on punctured elliptic curves
with boundary conditions as described in Corollary~\ref{cor:main} and
show that they coincide with elliptic KP spectral curves. It turns out
that Krichever's ansatz \cite{Kr} yields an algebraic approach to
spectral curves and Floquet functions of $\dbar$--operators on
punctured elliptic curves with the given boundary conditions.

\subsection*{Appendix to Section~\ref{sec:spec_minimal}}

The following lemma is needed in the proof of Theorem~\ref{the:main}.
 
\begin{Lem} \label{lem:integration_preserving_monodromy} Let $\omega$
  be a meromorphic 1--form with non--trivial monodromy $h\in
  \Hom(\Gamma,\C_*)$ and vanishing residues on a torus
  $T^2=\C/\Gamma$. Then there exists a unique meromorphic function $f$
  with monodromy $h$ on $T^2$ that satisfies $df=\omega$.
\end{Lem}

\begin{proof}
  Take $\tilde f$ an arbitrary meromorphic function with $d\tilde f =
  \omega$. Because $\omega_{z+\gamma} = \omega_z h_\gamma$ for
  $\gamma\in \Gamma$, there is $a_\gamma\in \C$ such that $\tilde
  f(z+\gamma)- \tilde f(z)h_\gamma = a_\gamma$ for all $z$. Now
  \[ \tilde f(z+\gamma_1 + \gamma_2) = \tilde f(z)h_{\gamma_1}
  h_{\gamma_2} + a_{\gamma_1} h_{\gamma_2} + a_{\gamma_2}\] and,
  because $\tilde f(z+\gamma_1 + \gamma_2) = \tilde f(z+\gamma_2 +
  \gamma_1)$ and $h$ is a representation, we obtain $a_{\gamma_1}
  h_{\gamma_2} + a_{\gamma_2} = a_{\gamma_2} h_{\gamma_1} +
  a_{\gamma_1}$ and hence 
  \begin{equation}
    \label{eq:ah}
    a_{\gamma_1}(h_{\gamma_2} -1 ) =
    a_{\gamma_2}(h_{\gamma_1} -1 ).    
  \end{equation}
  In particular, because $h$ is non--trivial, we have $a_\gamma=0$ for
  all $\gamma\in\Gamma$ such that $h_\gamma = 1$. On the other hand,
  adding a constant $b\in C$ to $\tilde f$ changes $a_\gamma$ to
  $a_\gamma + b(1-h_\gamma)$. For $\gamma\in \Gamma$ such that
  $h_\gamma\neq 1$ we define $b= \frac{a_\gamma}{h_\gamma-1}$. By
  \eqref{eq:ah} the definition of $b$ does not depend on the choice of
  $\gamma$ and $b$ is the unique solution to $a_\gamma +
  b(1-h_\gamma)=0$ for all $\gamma\in \Gamma$.  In particular, among
  the meromorphic functions satisfying $df=\omega$, the function
  $f=\tilde f + b$ is the unique one with multiplicative monodromy $h$
  (and no additional ``additive monodromy'').
\end{proof}

\section{Example: Minimal tori with four planar ends at the half
  periods}\label{sec:four_ends}
 
The smallest number of ends possible for Euclidean minimal tori with
planar ends is four, as shown by Kusner and Schmitt \cite{KS}.  This
is analogous to the case of minimal spheres with planar ends
\cite{Br84} which, except for the plane, have also at least four ends.
The first examples of Euclidean minimal tori with four planar ends are
given by Costa \cite{Co} (who treats rectangular tori) and Kusner and
Schmitt \cite{KS} and have ends located at the half periods.

In the following we show that the spectral curve of a minimal torus
with four planar ends located at the half periods is the double of the
spectral curve of a 1--gap Lam\'e potential.  As we will see, it is
natural to view the parameter domain of the Euclidean minimal torus as
a 4--fold covering of the spectral curve belonging to the Lam\'e
potential; the four ends of the minimal surface then cover the one end
of the Lam\'e spectral curve.  The spectral curve point of view yields
natural candidates for the spinors fields $s_1$ and $s_2$ needed to
construct the minimal immersion.

\begin{The}\label{the:four-ends}
  The spectral curve of a Euclidean minimal torus with four planar
  ends located at the half periods is the double
  $\Sigma=\Sigma'\, \dot \cup\, \bar \Sigma'$ of the spectral curve $\Sigma'$
  of a 1--gap Lam\'e potential.
\end{The}
\begin{proof}
  By Corollary~\ref{cor:main}, one can compute one component $\Sigma'$
  of the spectral curve $\Sigma=\Sigma' \, \dot \cup\, \bar \Sigma'$ by
  normalizing the 1--dimensional set of multipliers $h\in
  \Hom(\Gamma,\C_*)$ for which there exists a meromorphic function
  $\Phi$ with monodromy $h$, at most first order poles at the ends,
  and vanishing order zero term in the Laurent expansion at each
  end. The normalization of the set of possible $h$'s is an
  irreducible compact Riemann surface with one puncture.

  Denote by $T^2=\C/{\tilde \Gamma}$ with
  $\tilde\Gamma=\Span_\Z\{4\omega_1,4\omega_3\}$ a uniformization of
  the torus. Then $\Sigma'= (\C /\Gamma)\backslash \{0\}$ with $\Gamma
  = \frac12 \tilde \Gamma=\Span_\Z\{2\omega_1,2\omega_3\}$ is in a
  natural way a component of the spectral curve: for every $\alpha\in
  \Sigma'$, the 1--gap Lam\'e Baker--Akhiezer function $\Phi_\alpha$
  of $\C /\Gamma$, when seen as a function on the 4--fold covering
  $T^2=\C/\tilde \Gamma$ of $\C/\Gamma$, has exactly the right kind of
  Laurent expansion at the ends (see
  Appendix~\ref{sec:Krichever-BA}). The monodromy of $\Phi_\alpha$ on
  $\C/\Gamma$ in the direction $\gamma=2\omega_j$ is
  $e^{2(\zeta(\alpha)\omega_j - \alpha\eta_j)}$ so that its monodromy
  on the 4--fold covering $T^2$ in the direction $\tilde
  \gamma=4\omega_j$ is $h^{\tilde \gamma} = e^{4(\zeta(\alpha)\omega_j
    - \alpha\eta_j)}$. Thus, $\Sigma'$ is indeed a Riemann surface
  with one end that parametrizes a subset of the monodromies which, as
  in Corollary~\ref{cor:main}, are possible for meromorphic functions
  on $T^2$ with first order poles and vanishing order zero terms at the
  half periods.
\end{proof}

In the following we explain how to reconstruct all Euclidean minimal
tori with four planar ends at the half periods of $T^2=\C/{\tilde
  \Gamma}$ and spin structure corresponding to a given
$\Z_2$--multiplier $h_0\in \Hom(\tilde \Gamma,\Z_2)$. For this one has
to understand how to solve
\begin{itemize}
\item[a)] the \emph{algebraic problem} of finding a two linearly
  independent meromorphic functions $s_1$, $s_2$ with monodromy $h_0$
  as in Corollary~\ref{cor:spinors} (i.e., with poles of order at most
  one at the half periods, with vanishing order zero terms in the
  Laurent expansions at the half periods, and without common zeroes),
\item[b)] and the \emph{period problem} \eqref{eq:period_condition}.
\end{itemize}

One can check\footnote{In fact, for given $h_0\not \equiv 1$ the
  4--dimensional space of meromorphic functions on $T^2/\tilde \Gamma$
  with monodromy $h_0$ and first order poles at the half periods is
  spanned by translates of Baker functions $\Phi_\alpha$ on
  $T^2/\tilde \Gamma$, cf.\ Appendix~\ref{sec:Krichever-BA}, for
  $\alpha$ equal to one of the $\omega_i$, $i=1$,...,$3$ depending on
  $h_0$; but no linear combination of these translates has vanishing
  order zero terms in the Laurent expansion at all four half periods.}
that the algebraic problem cannot be solved for non--trivial spin
structure $h_0\not \equiv 1$.  In the case of trivial spin structure
$h_0\equiv 1$, there is a 3--dimensional space of meromorphic
functions on $T^2=\C/\tilde \Gamma$ whose only poles are first order
poles at the half periods and whose Laurent expansions at the half
periods have vanishing order zero terms. It is spanned by the
Baker--Akhiezer functions $\Phi_1=\Phi_{\alpha=\omega_1}$,
$\Phi_2=\Phi_{\alpha=\omega_2}$ and $\Phi_3=\Phi_{\alpha=\omega_3}$ on
$\C/\Gamma$ (see Appendix~\ref{sec:Krichever-BA}) viewed as functions
on the 4--fold covering $T^2=\C/\tilde\Gamma$, for $\Gamma =
\Span_\Z\{2\omega_1,2\omega_3\}$ and $\tilde \Gamma =
\Span_\Z\{4\omega_1,4\omega_3\}$.

Because the natural $\R_+\SU(2)$--action on $s_1$, $s_2$ changes the
immersion by a homothety only, the general ansatz for $s_1$ and $s_2$
of Euclidean minimal tori with four ends at the half periods and
trivial spin structure is
\begin{align}
  \label{eq:ansatz-four-end}
  s_1 & = \Phi_1 + a\, \Phi_2 + b\, \Phi_3 \\
  s_2 & = c\, \Phi_2 + d\, \Phi_3
\end{align}
(that $\Phi_1$ has a non--trivial coefficient can be achieved by
renumbering the basis of $\Gamma$ if necessary).

For the computation of the periods we use that
\begin{equation}
  \label{eq:squares-of-ba}
  \begin{split}
    \Phi_k(x)\Phi_l(x) & = \tilde \wp(x) \pm \tilde \wp(x-2\omega_1)\pm \tilde \wp(x-2 \omega_2) \pm \tilde \wp(x- 2\omega_3)    \qquad k\neq l \\
    \Phi_k^2(x)  & = \tilde \wp(x) + \tilde
    \wp(x-2\omega_1)+ \tilde \wp(x-2 \omega_2) + \tilde \wp(x-
    2\omega_3) - \wp(\omega_k),
  \end{split}
\end{equation}
where $\wp$ and $\tilde \wp$ denote the Weierstrass $\wp$--functions
of $\C/\Gamma$ and $\C/\tilde \Gamma$, respectively, and where in the
first equation always two of the $\pm$ are $-$--signs (e.g.\ for
$\Phi_1\Phi_2$ we have -,-,+).  As a consequence $\Phi_k\Phi_l\, dx$
with $k\neq l$ has no periods at all and
\begin{equation}
  \label{eq:period-four-point}
  \int_{\gamma=4\omega_l}\Phi_k^2\, dx = - 4( \eta_l + \wp(\omega_k) \omega_l)
  = - 4( \eta_l + e_k \omega_l). 
\end{equation}

Evaluation \eqref{eq:period_condition} along $\gamma=4\omega_1$ and
$\gamma=4\omega_3$ yields
\begin{equation}\label{ed:closedness1}
  \begin{split}
    \Re\big( ac(\eta_1 +e_2\omega_1) + bd (\eta_1 + e_3 \omega_1)\big)=0\\
    \Re\big( ac(\eta_3 +e_2\omega_3) + bd (\eta_3 + e_3
    \omega_3)\big)=0    
  \end{split}
\end{equation}
and
\begin{equation}\label{ed:closedness2}
  \begin{split}
    \overline{(\eta_1 +e_1\omega_1) + a^2(\eta_1 + e_2 \omega_1)+ b^2(\eta_1 +e_3\omega_1)}  = c^2(\eta_1 +e_2\omega_1) + d^2(\eta_1 + e_3 \omega_1) \\
    \overline{(\eta_3 +e_1\omega_3) + a^2(\eta_3 + e_2 \omega_3)+
      b^2(\eta_3 +e_3\omega_3)} = c^2(\eta_3 +e_2\omega_3) +
    d^2(\eta_3 + e_3 \omega_3).
  \end{split}
\end{equation}
A dimension count suggests that there is 2 real parameter space of
solutions.

Given $a$, $b$, $c$, $d$ such that \eqref{ed:closedness1} and
\eqref{ed:closedness2} are solved, one still has to check that $s_1$
and $s_2$ have no common zeros (so that one obtains an immersion).
The minimal torus with four planar ends is then given by plugging
$s_1$ and $s_2$ into \eqref{eq:classical_weierstrass} and
integrating. Because by \eqref{eq:squares-of-ba} one only has to
integrate $\wp$--functions on the parameter torus $T^2$, this yields
an explicit formula of the minimal surface in terms of Weierstrass
$\zeta$--functions on $T^2$.

\emph{Example:} setting $a=b=0$, \eqref{ed:closedness1} is solved
automatically and by matrix inversion \eqref{ed:closedness2} gives $c$
and $d$ uniquely up to sign, because
\begin{equation}
   \det  \begin{pmatrix}
    \eta_1 +e_2\omega_1 & \eta_1 + e_3 \omega_1 \\
    \eta_3 +e_2\omega_3 & \eta_3
  + e_3 \omega_3
\end{pmatrix}= (e_3-e_2) \det\begin{pmatrix}
  \eta_1 & \omega_1 \\
  \eta_3 & \omega_3
  \end{pmatrix} = (e_3-e_2) \frac{\pi i}2.
\end{equation}
The fact that the spinors $s_1$, $s_2$ generically have no common zeros
is checked in Section 23 of \cite{KS}. 

Examples of Euclidean minimal tori with more planar ends are
constructed in \cite{Sh}. It would be interesting to generalize our
construction of tori with four ends to an explicit construction, based
on elliptic soliton theory, of all Euclidean minimal tori with planar
ends.

Apart from our construction of minimal tori with four planar ends, in
minimal surface theory Baker--Akhiezer functions previously appeared
in Bobenko's paper \cite{Bob} on helicoids with handles.  In both
cases the spinors describing the minimal surface are Baker--Akhiezer
functions, so that the Euclidean minimal surfaces in question are
parametrized by (coverings of) their own spectral curves. This
relation between Euclidean minimal surfaces and integrable systems is
fundamentally different from the well established theory
\cite{PS89,Hi,Bob91} of constant mean curvature tori in $\R^3$ with
$H\neq 0$ (where the immersion is parametrized by a suitable real
torus in the Jacobian of the spectral curve).

\section{Surfaces with translational periods and Riemann's minimal
  surfaces}\label{sec:riemann}

Euclidean minimal tori with planar ends and translational periods can
be treated along the same lines as closed minimal tori with planar
ends. In the following we investigate the simplest
non--trivial\footnote{The trivial example here being the plane which
  can be viewed as a minimal torus with two translational periods and
  no ends on its fundamental domain.}  examples which are tori with
one translational period and two planar ends on each fundamental
domain. The embedded examples among these surfaces are Riemann's
``staircase'' minimal surfaces, see \cite{LRW}.  Like in
Section~\ref{sec:four_ends}, the spectral curves of Euclidean minimal
tori with translational periods and two planar ends are doubles of
1--gap Lam\'e spectral curves.

For a Euclidean minimal torus with planar ends and translational
periods, both the M\"obius invariant quaternionic holomorphic line
bundle $V/L$ (Appendix~\ref{sec:moebius-inv}) and the Euclidean
holomorphic line bundles $L$ and $KL^{-1}$
(Appendix~\ref{sec:weierstrass}) are well defined. The main difference
to the case with trivial translational monodromy is that the flat
bundle $V$ is no longer trivial. Moreover, the holomorphic section
$\psi_2$ is not periodic (instead $\psi_1$ and $\psi_2$ span a
two--dimensional linear system with monodromy; the monodromies are
$2\times2$ Jordan blocks with $1$ on the diagonal and translational
periods appearing in the upper right corners). The holomorphic section
$\psi_1$ of $V/L$ that corresponds to $\infty$ is still periodic,
because $\infty$ is fixed under translations. Also, the holomorphic
section $\psi$ appearing in the Weierstrass representation is well
defined.  In particular, the functions $s_1$, $s_2$ are well defined
meromorphic functions with $\Z_2$--monodromy and pole and zero
behavior as in Corollary~\ref{cor:spinors}. Only the periodicity
condition \eqref{eq:period_condition} is not satisfied anymore.

The spectral curve of a minimal torus with planar ends and
translational monodromy can still be defined as the spectral curve of
the quaternionic holomorphic line bundle $V/L$. In particular, the
proof of Theorem~\ref{the:main} goes through without change (the
non--periodic holomorphic section $\psi_2$ appears in the proof only
in local considerations about the pole behavior of $\Phi_1$ and
$\Phi_2$ and the smoothness of $\psi^h=\psi_1\chi$).

\begin{Lem}\label{lem:two-point}
  A Euclidean minimal torus with two planar ends and translational
  periods has non--trivial spin structure and ends at two of the four
  half periods of the torus.
\end{Lem}
\begin{proof}
  Let $T^2=\C/\Gamma$ with $\Gamma=\Span_\Z\{2\omega_1, 2\omega_3\}$ and
  assume that the ends are located at $0$ and $p\in T^2$.  In order to
  see that the spin structure is not trivial, note that a meromorphic
  function with trivial monodromy and first order poles at $0$ and
  $p\in T^2$ is of the form \[x\in T^2\mapsto a(\zeta(x)-\zeta(x-p)) +
  b\] with $a$, $b\in \C$. The condition that the order zero term in
  the Laurent expansion at both ends vanishes is $a \zeta(p) + b =
  0$. It is thus impossible to find two linearly independent functions
  $s_1$ and $s_2$ as in Corollary~\ref{cor:spinors} with trivial
  monodromy.

  Therefore, the spin structure has to be non--trivial and we can
  assume that the corresponding $\Z_2$--multiplier $h_0$ satisfies
  $h_0(2\omega_2) = 1$ and $h_0(2\omega_3)=-1$. The two--dimensional
  space of meromorphic functions with monodromy $h_0$ and first order
  poles at $0$ and $p$ is then spanned by the Baker functions $x\mapsto
  \Phi_{\alpha=\omega_2}(x)$ and $x\mapsto
  \Phi_{\alpha=\omega_2}(x-p)$ on $\C/\Gamma$ (see
  Appendix~\ref{sec:Krichever-BA}). The condition of
  Corollary~\ref{cor:spinors} that the order zero terms in the Laurent
  expansion at the ends vanishes for both sections is
  $\Phi_{\alpha=\omega_2}(p)=\Phi_{\alpha=\omega_2}(-p)=0$. But this
  only holds if $p=\omega_2\in T^2$.
\end{proof}

The proof of the following theorem is analogous to that of
Theorem~\ref{the:four-ends}. The fundamental domain of the torus with
two ends is now a double covering of the Lam\'e spectral curve; the
two ends cover the one end of the spectral curve.

\begin{The}\label{the:two-ends}
  The spectral curve of a Euclidean minimal torus with two planar ends
  and translational periods is the double $\Sigma=\Sigma'\,
  \dot\cup\,\bar \Sigma'$ of a 1--gap Lam\'e spectral curve $\Sigma'$.
\end{The}

\begin{proof}
  As in Theorem~\ref{the:four-ends}, we use Corollary~\ref{cor:main}
  to compute one component of the spectral curve.  Denote by
  $T^2=\C/\Gamma$ with $\Gamma=\Span_\Z\{2\omega_1,2\omega_3\}$ a
  uniformization of the torus such that the ends are located at $0$
  and $\omega_2=\omega_1+\omega_3$. Then $\Sigma'= (\C
  /\Gamma')\backslash \{0\}$ with $\Gamma'
  =\Span_\Z\{\omega_2,2\omega_3\}$ is one component of the spectral
  curve, because for every $\alpha\in \Sigma'$ the Baker--Akhiezer
  function $\Phi_\alpha$ on $\C /\Gamma'$ (see
  Appendix~\ref{sec:Krichever-BA}), when seen as a function on the
  2--fold covering $T^2=\C/\Gamma$ of $\C/\Gamma'$, has first order
  poles and Laurent expansions with vanishing order zero terms at the
  ends.
\end{proof}

In the rest of the section we sketch how to recover the classification
\cite{LRW} of Euclidean minimal tori with two parallel planar ends and
one translational period.  In particular, we explain how Riemann's
``staircase'' minimal surfaces appear in our setting.  The general
ansatz for spinors $s_1$ and $s_2$ belonging to minimal tori with two
planar ends that are parallel to the $jk$--plane and located at the
points $0$ and $\omega_2$ of $T^2=\C/\Gamma$ with
$\Gamma=\Span_\Z\{2\omega_1,2\omega_3\}$ is (because as in the proof
of Lemma~\ref{lem:two-point} the monodromy of $s_1$ and $s_2$ has to
be $h_0$ with $h_0(2\omega_2)=1$ and $h_0(2\omega_3)=-1$)
\[ s_1 = \Phi_1 + \Phi_2 \qquad \textrm{ and } \qquad s_2 = a(\Phi_1 -
\Phi_2) \] with $a\in \C$, where $\Phi_1$, $\Phi_2$ denote the Baker
functions $\Phi_{\alpha=\omega_2/2}$ and
$\Phi_{\alpha=\omega_2/2+\omega_3}$ on $\C/\Gamma'$ (see Appendix~\ref{sec:Krichever-BA}) with $\Gamma'
=\Span_\Z\{\omega_2,2\omega_3\}$ seen as functions on the double
covering $T^2=\C/\Gamma$.  The spinors $s_1$, $s_2$ have no common
zeros.  For integrating \eqref{eq:classical_weierstrass} and computing
the periods we use that (by comparing poles and zeros)
\begin{equation}
  \label{eq:squares-of-b2}
  \begin{split}
    \Phi_1(x)\Phi_2(x) & = \wp(x) - 
    \wp(x-\omega_2)    \\
    \Phi_j^2(x)   &= \wp(x) + \wp(x-\omega_2) - b_j,\qquad j=1,2
  \end{split}
\end{equation}
with $b_1=2\wp(\omega_2/2)$ and $b_2=2\wp(\omega_2/2+\omega_3)$. Thus,
as in Section~\ref{sec:four_ends}, \eqref{eq:classical_weierstrass}
can be explicitly integrated in terms of $\zeta$--functions.  Because
$\Phi_1 \Phi_2dx$ has no periods, the closedness
\eqref{eq:period_condition} of the minimal immersion described by
$s_1$ and $s_2$ in the direction $\gamma\in \Gamma$ reads
\[ a \int_\gamma (\Phi_1^2(x)-\Phi_2^2(x))dx \in i\R \quad \textrm{
  and } \quad \int_\gamma (\Phi_1^2(x)+\Phi_2^2(x))dx = {\bar a}^2
\overline{\int_\gamma (\Phi_1^2(x)+\Phi_2^2(x))dx}.\] The
period in the direction $\gamma_{m,n}=2(m\omega_1+n\omega_3)$,
$m,n\in \Z$ vanishes if
\begin{equation}
  \label{eq:closed-2end-1}
  2a (b_2-b_1) (m\omega_1+n\omega_3) \in i\R
\end{equation}
with
\begin{equation}
  \label{eq:closed-2end-2}
  \bar a= \pm \arg\big( 8(m\eta_1+n\eta_3) - 2 (b_1+b_2)
(m\omega_1+n\omega_3)\big)
\end{equation}
for $\arg(z)=z/|z|$.  This suggests that for given $m$, $n$ there is a
1--parameter family of points in the moduli space of genus one curves
for which the $\gamma_{m,n}$--period is closed, i.e.,
\eqref{eq:closed-2end-1} is satisfied for $a$ prescribed via
\eqref{eq:closed-2end-2} by $m$, $n$. A more detailed investigation of
this would recover the classification given in Theorem~3.2 of
\cite{LRW}.

The embedded examples in the moduli space of all Euclidean minimal
tori with two parallel planar ends are Riemann's ``staircase'' minimal
surfaces (see Theorem~3.1 of \cite{LRW}). They are parametrized over
rectangular tori and appear in our setting for $\omega_1\in \R$ and
$\omega_3\in i\R$. Because this implies $b_2=\bar b_1$, for
$\gamma_{m,n}$ with $mn=0$ and the corresponding choice
\eqref{eq:closed-2end-2} of $a$, equation \eqref{eq:closed-2end-1} is
then automatically satisfied. The parametrizations of Riemann's
minimal surfaces thus obtained essentially coincide with the ones
given in \cite{HKR}.

Riemann's minimal surfaces have previously been characterized using
algebro geometric KdV theory in the work of Meeks, Perez, and Ros
\cite{MPR,MP}. It would be interesting to conceptually relate the way
KdV theory appears in their work with our approach based on spectral
curve theory.

\appendix

\section{Some notions of quaternionic holomorphic
  geometry}\label{app:quat_hol}

We review the relevant notions of quaternionic holomorphic geometry
\cite{PP,FLPP01,BFLPP02}. In particular we discuss the quaternionic
holomorphic approach \cite{BPP09,BLPP} to the spectral curve of conformally
immersed tori \cite{Ta98,GS98,Ta04,Ta06}.

\subsection{Quaternionic holomorphic line bundles}\label{app:linebundles}
A \emph{quaternionic holomorphic line bundle} over a Riemann surface
$M$ is a quaternionic line bundle $L$ equipped with a complex
structure $J\in \Gamma(\End L)$, $J^2=-\Id$ and a quaternionic linear
(Dirac type) differential operator $D\colon \Gamma(L) \to \Gamma(\bar
KL)$ satisfying the Leibniz rule
\begin{equation*}
  D(\psi \lambda) = (D \psi)\lambda + (\psi d\lambda)''
\end{equation*}
for all $\psi \in \Gamma(L)$ and $\lambda\colon M \to \H$, where
$\omega'':=\frac12(\omega + J{*}\omega)$ and $\bar K L$ denotes the
bundle of 1--forms with values in $L$ that transform like $*\omega =
-J \omega$ for $*$ the induced complex structure on $T^*M$. The degree
of a quaternionic holomorphic line bundle is defined as the degree of
the underlying complex line bundle $E := \{ \psi \in L \mid J\psi
=\psi \ii\}$.  The complex structure $J$ decomposes the operator
$D=\dbar+Q$ into $J$--commuting part $\dbar$ and anti--commuting part~$Q$.  The operator $\dbar$ respects the complex line bundle $E$ and
defines a complex holomorphic structure.  The tensor field $Q$ is
called the \emph{Hopf field} of the quaternionic holomorphic line
bundle.

There are two essentially different ways how quaternionic holomorphic
line bundles arise in the theory of conformal immersions of Riemann
surfaces into 4--space. One of them is M\"obius invariant and can be
best understood within the quaternionic model of the conformal
4--sphere $S^4=\HP^1$; the other depends on the choice of a Euclidean
subgeometry or, more precisely, on the choice of a point at infinity
$\infty \in S^4=\HP^1$.

By trivialising the bundle with a $\dbar$--holomorphic section one can
bring the operator $D$ to the form of a Dirac operator
\eqref{eq:dirac-op} acting on functions. When dealing with
quaternionic holomorphic line bundles related to immersed surfaces in
3--space, such a trivialization can always be achieved globally if one
uses sections and functions with $\Z_2$--monodromy which takes into
account the spin structure of the immersion.

\subsection{M\"obius invariant representation}~\label{sec:moebius-inv}
In the following we identify maps $f\colon M \rightarrow S^4= \HP^1$
into the conformal 4--sphere with quaternionic line subbundles
$L\subset V$ of a trivial quaternionic rank 2 bundle $V$ over $M$
equipped with a trivial connection.\footnote{Here $V$ can be simply
  thought of as the Cartesian product of $M$ with a quaternionic rank
  2 vector space. The reason for preferring the language of bundles
  and connections is that is simplifies the treatment of immersions
  with M\"obius monodromy, e.g.\ in Section~\ref{sec:riemann} where we
  discuss minimal tori with translational periods.}  If $f$ is a
conformal immersion, the quaternionic line bundle $V/L$ carries a
unique quaternionic holomorphic structure for which all sections
obtained by projection from constant sections of $V$ are holomorphic
(see the rest of the section for a more explicit coordinate
description and e.g.~\cite{BLPP} for more details). This holomorphic
structure on $V/L$ is M\"obius invariant, because its definition is
projectively invariant.

The conformal immersion is now encoded in the 2--dimensional linear
system of holomorphic sections of $V/L$ obtained by projection from
constant sections of $V$. The quotient of any two linearly independent
sections $\psi_1$, $\psi_2$ in this linear system is a coordinate
representation of the immersion in an affine chart of $\HP^1$. To see
this note that the choice of $\psi_1$, $\psi_2$ identifies $V$ with
the trivial $\H^2$ vector bundle with trivial connection. The
holomorphic sections $\psi_1$, $\psi_2$ of $V/L$ are then the
projections under $\pi \colon V\rightarrow V/L$ of the basis vectors
$e_1=\left(
\begin{smallmatrix}
  1\\0
\end{smallmatrix}\right)$ and $e_2=\left(
\begin{smallmatrix}
  0\\1
\end{smallmatrix}\right)$.
Away from the isolated points at which $f$ goes through 
$\infty=\left(
\begin{smallmatrix}
  1\\0
\end{smallmatrix}\right)\H$, the line bundle corresponding to the immersion can
be written as $L=\left(
\begin{smallmatrix}
  f\\1
\end{smallmatrix}\right)\H$, where $f\colon M \rightarrow
\H=\HP^1\backslash\{\infty\}$ is the representation of the immersion in the
affine chart defined by $\infty$. In particular, the immersions has the
quotient representation
\[ \psi_2= -\psi_1 f.\] Note that as long as the immersion does not go
through $\infty$, the section $\psi_1$ has no zeroes and the affine
representation $f$ is a globally smooth map $f\colon M \rightarrow
\H$. An arbitrary holomorphic section of $V/L$ then takes the from
$\psi_1 g$ for $g\colon M \rightarrow \H$ a function satisfying
$*dg=N\, dg$, where $N\colon M \rightarrow S^2\subset \Im(\H)$ is the
so called left normal of $f$ defined by $*df=N\, df$. In the case that
$f$ takes values in $\R^3=\Im(\H)$, the map $N$ is the Gauss map of
the immersion.

Given an immersion into the conformal 3-- or 4--sphere, a generic choice of
a point $\infty$ at infinity will avoid that the immersion goes through
$\infty$.  However, it might be preferable (for example in
Section~\ref{sec:spec_minimal}) to chose a point $\infty$ at infinity for
which the immersion does go through $\infty$.  In this case, the affine
representation $f$ is smooth away from the \emph{ends} at which it goes
through $\infty$. At the ends the quaternion valued function $f$ then has
first order poles in the sense that $f^{-1}$ vanishes to first order: it
vanishes, because $f$ has an end, and its vanishing order is one, because $f$
is the affine representation of an immersion into the conformal 4--sphere.

\subsection{Euclidean Weierstrass representation}\label{sec:weierstrass}
Give a conformal immersion $f\colon M \rightarrow \R^4=\H$ of a Riemann
surface into Euclidean 4--space, there are unique (up to isomorphism)
quaternionic holomorphic line bundles $L$ and $\tilde L$ with holomorphic
sections $\psi$ and $\alpha$ and a unique pairing \cite[Section 2.3]{FLPP01}
between $\tilde L$ and $L$, i.e., a quaternionic sesquilinear map $(,)\colon
\tilde L \times L \rightarrow T^*M \otimes \H$, such that
\[ (\alpha,\psi) = df.\] This is the quaternionic version \cite{PP} of
the Weierstrass representation of $f$ and the bundle $\tilde L$ is
isomorphic to $KL^{-1}$. Although the quaternionic holomorphic
structures on the line bundles $L$ and $KL^{-1}$ are not M\"obius
invariant, the underlying paired quaternionic line bundles and their
complex holomorphic structures $\dbar$ are M\"obius invariant.

If $f$ takes values in $\R^3=\Im(\H)$, then $\psi\mapsto \alpha$
defines an isomorphism between $L$ and~$\tilde L$.  The bundle $L$ is
then called a \emph{quaternionic spin bundle}, the quaternionic
holomorphic structure is compatible with the pairing in the sense that
holomorphic sections square to closed forms, and $f$ has Weierstrass
representation $df=(\psi,\psi)$. In particular, the underlying complex
bundle $E$ with $\dbar$ is then a complex holomorphic spin bundle (a
square root of the canonical bundle), because if $\varphi$ is a
$\dbar$--holomorphic section of $E$, the pairing $(\varphi,\varphi)$
is $j$ times a holomorphic differential.

We explain now, how the Weierstrass representation fits into the
quaternionic projective picture explained above. First note that the
projective differential $\delta:=\pi \nabla_{|L}$ of a conformal
immersion $f\colon M \rightarrow S^4= \HP^1$ is a bundle isomorphism
$\delta \colon L \rightarrow KV/L$. This allows to define a M\"obius
invariant pairing \cite[Section 2.3]{FLPP01} between $L$ and the
bundle $L^\perp \subset V^*$ perpendicular to $L$ by setting
\[ (\alpha,\psi) := \, < \alpha, \delta \psi > = \,- < \delta^\perp \alpha,
\psi>,\] where $\delta^\perp\colon L^\perp \rightarrow KL^{-1}=KV^*/L^\perp$
denotes the projective derivative of $L^\perp$. 

From now on assume that we have fixed a point $\infty \in \HP^1$.
Away from the ends of the immersion $f$, i.e., the isolated points at
which $f$ goes through $\infty$, the point $\infty$ defines
holomorphic structures on the bundles $L$ and $L^\perp$: denote by
$\psi_1\in H^0(V/L)$ a holomorphic section of $V/L$ obtained by
projecting a non--trivial, constant section of the quaternionic line
bundle $\infty \subset V$.  Away from the ends (which coincide with the
vanishing locus of $\psi_1$), there is a unique flat connection $\nabla$
on $V/L$ satisfying $\nabla \psi_1=0$. This connection defines
holomorphic structures $\nabla''$ on $L^\perp =(V/L)^{-1}$ and
$d^\nabla$ on $KV/L\cong L$.

We show now that these holomorphic structures are the holomorphic
structures occurring in the Weierstrass representation of the
corresponding immersion into the Euclidean space $\H = \HP^1\backslash
\{\infty\}$: let $f\colon M \rightarrow \H\subset \HP^1$ so that the
corresponding quaternionic line subbundle is $L=\psi \H$ with $\psi=
\left(\begin{smallmatrix} f\\1
  \end{smallmatrix}\right)$ and $\infty = \left(\begin{smallmatrix}
    1\\0
\end{smallmatrix}\right)\H$. The section $\psi$ is holomorphic since $\delta\psi = \psi_1
df\in \Gamma(KV/L)$ and holomorphic sections of $KV/L$ are precisely the
sections of the form $\psi\eta$ with $*\eta=N\eta$ and $d\eta=0$. The section
$\alpha\in \Gamma(L^\perp)$ defined by $<\alpha,\psi_1>=1$ is holomorphic as
well and $(\alpha,\psi)= df$ which is the Weierstrass representation of $f$.

\subsection{The spectral curve of a quaternionic holomorphic line
  bundle}~\label{sec:spectral_curve} A conformal invariant attached to
an immersion of a torus into 3--space is its \emph{spectral curve}
\cite{Ta98,GS98,Ta04,Ta06,BPP09,BLPP}.  In
Appendix~\ref{sec:spec-immersion} we define the spectral curve of an
immersed torus based on the notion of spectral curve for quaternionic
holomorphic line bundles of degree~$0$ on the torus which we discuss
here. This spectral curve can be equivalently viewed as the spectral
curve of a periodic 2--dimensional Dirac operator~\eqref{eq:dirac-op},
cf.\ Appendix~\ref{app:linebundles}.

Following \cite{BPP09}, we define the spectral curve $\Sigma$ of a
quaternionic holomorphic line bundle $L$ of degree $0$ over a
torus~$T^2=\C/\Gamma$ as the Riemann surface normalizing the complex
analytic set of possible Floquet multipliers (or monodromies) of
non--trivial holomorphic sections of $L$, i.e., the set of \[h\in
\Hom(\Gamma,\C_*)\cong \C_* \times \C_*\] for which there exists a
non--trivial solution to $D\psi = 0$ defined on the universal covering
$\C$ of $T^2=\C/T^2$ that transforms according to
\[ \psi(x+\gamma) = \psi(x) h_\gamma\] for all $x\in T^2$ and
$\gamma\in \Gamma$. In order to justify the definition of the spectral
curve $\Sigma$ one has to verify that the possible multipliers form a
1--dimensional complex analytic set. In \cite{BPP09} this is proven by
asymptotic analysis of a holomorphic family of elliptic operators. In
addition it is shown that $\Sigma$ has one or two ends (depending on
whether its genus is infinite or finite) and one or two connected
components each containing at least one end. Moreover, for a generic
Floquet multiplier $h$ that admits a non--trivial holomorphic section,
the space of holomorphic sections with monodromy $h$ is complex
1--dimensional.

\subsection{The spectral curve of a conformally immersed torus}\label{sec:spec-immersion}
For an immersed torus in 4--space whose normal bundle is topologically
trivial, the spectral curve can be defined using either the M\"obius
invariant or the Euclidean quaternionic holomorphic line bundles
attached to the immersion. The following theorem shows that both
possibilities lead to the same Riemann surface which is therefore
M\"obius invariant.\footnote{It should be noted that when the spectral
  curve is defined using the Euclidean concept of Weierstrass
  representation as in its original definition~\cite{Ta98} for tori in
  $\R^3$, its M\"obius invariance is far from obvious. It was first
  conjectured by the second author \cite{Ta97} and first proven,
  independently, by Grinevich and Schmidt \cite{GS98} and by Pinkall
  (unpublished).  A non--normalized (more precise) version of the
  spectral curve is considered in \cite{GT} where it is shown that
  M\"obius transformations of tori may result in creation and
  annihilation of multiple points on the corresponding spectral curve.
  The proof here is a variant of Pinkall's unpublished proof.}

\begin{The}\label{the:moebius_inv_of_spec}
  For a conformal immersion $f\colon T^2 \rightarrow \R^4=\H\subset
  \HP^1$ with topologically trivial normal bundle, the spectral curve
  of the M\"obius invariant quaternionic holomorphic line bundle $V/L$
  coincides with the spectral curve of the Euclidean quaternionic
  holomorphic structure on $L$.
\end{The}

\begin{proof}
  Because of the asymptotics of spectral curves, i.e., the fact that
  they have at most two components each of which contains an end
  \cite{BPP09}, it is sufficient to check that the set of possible
  monodromies of holomorphic sections of $V/L$ is contained in the set
  of possible monodromies of holomorphic sections of $L=KV/L$ equipped
  with the Euclidean holomorphic structure $d^\nabla$, where $\nabla$
  is the flat connection on $V/L$ defined by the point $\infty$. But
  this immediately follows from the fact that, if $\psi^h$ is a
  holomorphic section with monodromy $h$ of $V/L$, by flatness of
  $\nabla$ its derivative $\nabla\psi^h$ is a holomorphic section of
  $KV/L$ which obviously has the same monodromy as $\psi^h$.
\end{proof}

Note that in Theorem~\ref{the:moebius_inv_of_spec} we assume that $f$
does not go through the point $\infty\in \HP^1$ defining the Euclidean
geometry, that is, the corresponding immersion into $\R^4=\H$ is
assumed to have no ends.  In contrast to this, in the proof of
Theorem~\ref{the:main} we choose $\infty$ for which $f$ has ends.

\section{Conformally immersed tori with reducible spectral curve and elliptic KP solitons}\label{app:elliptic_KP}

The spectral curve $\Sigma$ of a conformally immersed torus
(Appendix~\ref{sec:spec-immersion}) or, more generally, a quaternionic
holomorphic line bundle on a torus (Appendix~\ref{sec:spectral_curve})
is equipped with a pair of holomorphic maps
\[ h_j\colon \Sigma \rightarrow \C_*\qquad j=1,2 \] that describe the
monodromies of holomorphic sections in the direction of a basis
$\gamma_1$, $\gamma_2$ of the lattice $\Gamma$ defining the underlying
torus $T^2=\C/\Gamma$.  If $\Sigma$ has finite genus, the logarithmic
derivatives \[\eta_j=d\log(h_j)\qquad j=1,2\] are meromorphic 1--forms
on the compactification $\bar\Sigma=\Sigma\cup \{o,\infty\}$ of
$\Sigma$ with second order poles and no residues at the ends $o$,
$\infty$ (see e.g. Lemma~5.1 of \cite{BPP09}).  Moreover, all periods
of $\eta_j$, $j=1,2$ take values in $2\pi i\Z$. The existence of a
pair of meromorphic forms with the given asymptotics and periodicity
is a ``closing condition'' for the underlying Dirac potential. In
integrable surface theory, a closing condition of this type probably
first appeared in \cite{Hi} (see Theorem 8.1 there).

The same kind of closing condition characterizes compact Riemann
surfaces with one puncture that arise as spectral curves of elliptic
KP solitons:

\begin{Pro}
  A compact Riemann surface with one puncture is the spectral curve of
  an elliptic KP soliton if and only if it admits two linearly
  independent meromorphic 1--forms $\eta_1$, $\eta_2$ with single
  second order poles and no residues at the puncture such that all
  periods are in $2\pi i\Z$.
\end{Pro}

\begin{proof}
  We only sketch the proof, without going into details of the
  underlying finite gap integration theory: a compact Riemann surface
  with one puncture gives rise, via the Krichever construction, to a
  function in $x$, $y$, $t$ which is a meromorphic solution to the KP
  equation. This solution can be obtained from the Baker--Akhiezer
  function which is characterized as the section of a family of
  holomorphic line bundle on the Riemann surface. Its dynamics in $x$,
  $y$, $t$ is given by linear flows in the Jacobian, realized by
  cocycles which linearly depend on $x$, $y$, $t$ and describe the
  change of holomorphic line bundle.  An elliptic KP soliton as
  described in \cite{Kr} is precisely a finite gap KP solution that is
  double periodic in the $x$--variable.  As in the case with two
  punctures (see e.g.\ p.665 in \cite{Hi}), the existence of a pair of
  1--forms $\eta_1$, $\eta_2$ with the given asymptotics and
  periodicity means that there is a pair of $x$--values whose cocycles
  are coboundaries, i.e., that the dynamics in the $x$--direction is
  periodic.
\end{proof}

If the spectral curve of an immersed torus is reducible, the
closedness condition on the full two--punctured spectral curve implies
the closedness condition on each component. In particular, by the
preceding proposition, each component is an elliptic KP spectral
curve.

\begin{Cor}
  If the spectral curve of a conformally immersed torus is reducible,
  its components are spectral curves of elliptic KP solitons.
\end{Cor}

It would be interesting to understand whether in the special case of
immersed tori in the conformal 3--sphere, reducible spectral curves
are always (like in the examples discussed in
Sections~\ref{sec:four_ends} and \ref{sec:riemann}) spectral curves of
elliptic KdV solitons, i.e., elliptic KP spectral curves for which the
KP solitons don't depend on the $y$--variable and hence solve the KdV
equation.

\section{Elliptic functions and Lam\'e potentials}\label{app:elliptic}

We collect some facts about elliptic functions and Baker--Akhiezer
functions for Hill's equation with 1--gap Lam\'e potential.

\subsection{Weierstrass's elliptic functions} For the uniformization of a
conformal 2--torus we write $T^2 = \C/ \{2\omega_1,2\omega_3\}$.  The
Weierstrass $\wp$--function is the unique elliptic function on $T^2$ with a
single pole of order two and the asymptotics $\wp(x)=\frac1{x^2} + O(x^2)$;
its other three branch points are the half lattice vectors $\omega_1$,
$\omega_3$ and $\omega_2=\omega_1+\omega_3$ and \[(\wp')^2 =
4(\wp(x)-e_1)(\wp(x)-e_2)(\wp(x)-e_3)\] for $e_j=\wp(\omega_j)$, $j=1,2,3$.
The Weierstrass $\zeta$--function is the unique function satisfying
$\zeta'=-\wp$ with $\zeta(x) = \frac1x + O(x^3)$.  Because $\wp$ is even,
$\zeta$ is odd and its translational periods $\zeta(x+2\omega_j) = \zeta(x) +
2\eta_j$ are given by $\eta_j=\zeta(\omega_j)$, $j=1,2,3$.  In our notation
the Legendre relation reads
\[ \eta_1\omega_3 - \eta_3 \omega_1 = \frac{\pi i}2.\]
The Weierstrass $\sigma$--function is the unique solution to $\sigma' = \zeta
\sigma$ with the asymptotics $\sigma(x) = x+O(x^5)$. Because $\zeta$ is odd,
$\sigma$ is odd. Its monodromy is given by 
\[ \sigma(x+2\omega_j) = -\sigma(x) e^{2\eta_j(x+\omega_j)}, \qquad j=1,2,3.\]
The entire function $\sigma$ satisfies
\[ \sigma(x) = x \prod_{(m,n)\neq 0} \left(1- \frac
  x{\omega_{m,n}}\right)\exp\left(\frac x{\omega_{m,n}} + \frac12 \left(
    \frac x{\omega_{m,n}}\right)^2 \right)\] with $\omega_{m,n} = 2m\omega_1
+ 2n \omega_3$.

\subsection{Baker function for Hill's equation with Lam\'e
  potential}~\label{sec:Krichever-BA} The Baker--Akhiezer function of
Hill's equation
\[ \Phi''(x) -2\wp(x)\, \Phi(x) = E\, \Phi(x) \] with Lam\'e potential $-2\wp$
and spectral parameter $E=\wp(\alpha)$ is given by
\[ \Phi_\alpha(x) = \frac{\sigma(\alpha-x)}{\sigma(\alpha)\sigma(x)}
e^{\zeta(\alpha)x} \] for $0\neq \alpha \in
\C/\{2\omega_1,2\omega_3\}$, see \cite{Kr} or Chapter XXIII, 23$\cdot$\!\!
7 of \cite{WW} (according to which this formula goes back to Hermite
1877 and Halphen 1888).  In particular, the spectral curve of Hill's
operator with potential $-2\wp(x)$ is the elliptic curve on which the
potential is defined with the point $0$ removed; the Weierstrass
$\wp$--function describes the 2:1--map $\alpha\mapsto E= \wp(\alpha)$
from the spectral curve to the spectral parameter plane and the
hyperelliptic involution is $\alpha\mapsto -\alpha$.

The asymptotics of $\Phi_\alpha$ is
\[ \Phi_\alpha(x) = \frac1x -\frac12 \wp(\alpha) x + \frac16 \wp'(\alpha) x^2
+ ...\] and its monodromy is given by
\[ \Phi_\alpha(x+2\omega_j) = \Phi_\alpha(x) e^{2(\zeta(\alpha)\omega_j -
  \alpha\eta_j)}, \qquad j=1,2,3. \] As expected from the general spectral
theory of Hill's equation, the Baker--Akhiezer functions $\Phi_\alpha$
corresponding to the branch points $\alpha=\omega_j$, $j=1,2,3$ of $\wp$ have
$\Z_2=\{\pm 1\}$--monodromy (following in our case directly from the Legendre
relation).

\end{document}